\documentclass[12pt]{article}
\usepackage{graphicx}
\usepackage{amssymb}
\usepackage{pb-diagram}
\usepackage{amscd}
\usepackage{geometry} 
\usepackage{amsmath}
\usepackage{hyperref}

 \usepackage{float}
\usepackage{comment}
\usepackage{tikz}
\usepackage{amsthm}
\usepackage{enumerate}

\numberwithin{equation}{section}

\theoremstyle{plain}
\newtheorem{Th}{Theorem}[section]
\newtheorem{Lem}[Th]{Lemma}

\newtheorem{Cor}[Th]{Corollary}

\newtheorem{Prop}[Th]{Proposition}

\newtheorem{problem}[Th]{Problem}

\theoremstyle{remark}

\theoremstyle{definition}
\newtheorem{Def}[Th]{Definition}

\newcommand{\coh}{\mathrm{coh}}
\newcommand{\Ext}{\mathrm{Ext}}
\newcommand{\Hom}{\mathrm{Hom}}
\newcommand{\cO}{\mathcal{O}}
\newcommand{\cD}{\mathcal{D}}
\newcommand{\K}{K_\mathbb{C}}
\newcommand{\Enkr}{\widehat {\mathbb{E}}_k^n}
\newcommand{\Enk}{{\mathbb{E}}_k^n}

\newcommand{\dbcoh}{\cD(X)} 
\newcommand{\mT}{\mathcal{T}}
\newcommand{\mA}{\mathcal{A}}
\newcommand{\mB}{\mathcal{B}}
\newcommand{\mC}{\mathcal{C}}

\newcommand{\Z}{\mathbb{Z}}
\newcommand{\C}{\mathbb{C}}
\newcommand{\be}{\bar{E}}

\newcommand{\wa}{\widehat \mA}
\newcommand{\wb}{\widehat \mB}

\begin{document}

\title{Lefschetz exceptional collections\\[1ex]in $S_k$-equivariant categories of $(\mathbb{P}^n)^k$}
\author{Mikhail Mironov}
\date{ }

\clearpage
\maketitle

\begin{abstract}
 We consider the bounded derived category of $S_k$-equivariant coherent sheaves on $(\mathbb{P}^n)^k$. The goal of this paper is to construct in this category a rectangular Lefschetz exceptional collection when this is possible, or a minimal Lefschetz exceptional collection when a rectangular one does not exist. The main results of the paper include the construction of a rectangular Lefschetz exceptional collection in the case $k=3$ and in the case $n=1$ when $\mathrm{gcd}(n+1,k)=1$. We also construct minimal Lefschetz exceptional collection for $n=1$ and even $k$, and for $n=2$ and $k=3$.

\end{abstract}

\tableofcontents

\section {Introduction}

The bounded derived category of coherent sheaves is the main homological invariant of an algebraic variety
which captures the most essential geometric information.
It stands in the focus of many recent research papers.
One of the ways to describe it is via an exceptional collection.

Recall that an object $E$ in a $\mathbb{C}$-linear triangulated category $\mathcal{T}$ is \emph{exceptional} 
if~$\Ext^0(E,E) = \mathbb{C}$ and $\Ext^i(E,E) = 0$ for $i \ne 0$.
Furthermore, a collection~$E_1, \dots, E_r$ of objects in $\mathcal{T}$ is an \emph{exceptional collection}
if each $E_i$ is an exceptional object and~$\Ext^\bullet(E_i,E_j) = 0$ for $i > j$.
An exceptional collection is \emph{full} if the smallest full triangulated subcategory of $\mathcal{T}$ containing all $E_i$ 
coincides with $\mathcal{T}$.

Recently a special class of exceptional collections attracted much attention.
Recall that an exceptional collection $E_1,\dots,E_r$ in the bounded derived category of coherent sheaves $\dbcoh$ of a smooth projective variety $X$ is \emph{Lefschetz} with respect to a line bundle $\mathcal{L}$
if there is a partition $r = r_0 + r_1 + \dots + r_d$ with $r_0 \ge r_1 \ge \dots \ge r_d$ such that 
\begin{equation*}
E_{r_0 + r_1 + \dots + r_{i-1} + t} \cong E_t \otimes \mathcal{L}^i\qquad\text{for all } 1 \le t \le r_{i}\qquad\text{and } 1 \le i \le d.
\end{equation*}
In other words, if the objects of the collection are obtained by $\mathcal{L}$-twists from the subcollection of the first $r_0$ objects 
according to the pattern provided by the partition.

As it is clear from the definition, a Lefschetz collection is determined by its \emph{starting block} $E_1,\dots,E_{r_0}$ and the partition $(r_0,r_1,\dots,r_d)$.
It is less evident, but is still true, that if a Lefschetz collection is full, then the partition is itself determined by the starting block of the collection {\cite[Lemma~4.5]{Kuznetsov1}}.
Thus, extendability to a Lefschetz collection is just a property of an exceptional collection $E_1,\dots,E_{r_0}$.

It follows that there is a natural partial order on the set of all Lefschetz collections in $\dbcoh$ --- 
a Lefschetz collection with a starting block $E_1,\dots,E_{r_0}$ is \emph{smaller} than a Lefschetz collection with a starting block $E'_1,\dots,E'_{s_0}$
if $E_1,\dots,E_{r_0}$ is a subcollection in  $E'_1,\dots,E'_{s_0}$, see \cite[Definition~1.4]{KuznetsovSmirnov}.

A Lefschetz collection $E_1,\dots,E_r$ with partition $r_0,r_1,\dots,r_d$ is called \emph{rectangular of length $d+1$}, if $r_0 = r_1 = \dots = r_d$
(equivalently, if the Young diagram representing the partition is a rectangle of length $d+1$).
Of course, a necessary condition for the existence of a rectangular Lefschetz collection in $\dbcoh$ is a factorization 
\begin{equation}
\label{eq:k0-divisibility}
\operatorname{\mathrm{rk}} \big( K_0(\dbcoh) \big) = r_0(d + 1)
\end{equation}
for the rank of the Grothendieck group of $X$.
On the other hand, if a rectangular Lefschetz decomposition in $\dbcoh$ exists, and if its length~\mbox{$d + 1$} has the property that~$\mathcal{L}^{d+1} \cong \omega_X^{-1}$ where $\omega_X$ is the canonical bundle of $X$, 
that is $d + 1$ equals the \emph{index} of $X$ with respect to $\mathcal{L}$, 
then this collection is automatically minimal (this follows easily from Serre duality, see \cite[Subsection~2.1]{KuznetsovSmirnov}).

Lefschetz collections have many nice properties and are very important for homological projective duality and categorical resolutions of singularities {\cite{Kuznetsov2}}.
Especially nice and important are rectangular (resp.\ minimal) Lefschetz collections.
So, the following problem is very interesting.

\begin{problem}
\label{problem:lefschetz}
Given a smooth projective variety $X$ and a line bundle $\mathcal{L}$, construct a full rectangular Lefschetz collection in $\dbcoh$ 
with respect to $\mathcal{L}$ of length equal to the index of $X$, or, if the above is impossible, a minimal Lefschetz collection.
\end{problem}

There are many varieties $X$ for which the above problem was solved.
Among these are projective spaces, most of the Grassmannians, and some other homogeneous spaces {\cite{Fonarev}}.
In this paper we discuss Problem~\ref{problem:lefschetz} for a very simple variety
\begin{equation*}
X  = X_k^n :=\underbrace{\mathbb {P}^n \times \mathbb {P}^n \times \dots \times \mathbb {P}^n}_{k \text{ copies}},
\end{equation*}
but replace the category $\cD(X_k^n)$ with the equivariant derived category $\cD_{S_k}(X_k^n)$ 
with respect to the natural action of the symmetric group $S_k$ (by permutation of factors).
Note that this category can be considered as the derived category of the \emph{quotient stack}~$[X_k^n/S_k]$.
The line bundle $\mathcal{L}$ here is, of course, the ample generator~$\cO(1,1,\dots,1)$ of the invariant Picard group $\operatorname{\mathrm{Pic}}(X_k^n)^{S_k}$.
Note that the index of~$X_k^n$ with respect to $\mathcal{L}$ is equal to $n + 1$, so the goal of the paper can be formulated as follows.

\begin{problem}
\label{problem:xkn}
Find a full rectangular Lefschetz collection of length $n + 1$ in $\cD_{S_k}(X_k^n)$ with respect to the line bundle $\cO(1,1,\dots,1)$ 
or a minimal Lefschetz collection if the above is impossible.
\end{problem}

Note that without passing to the equivariant category the problem becomes trivial.
To construct a rectangular Lefschetz collection in $\cD(X_k^n)$ 
one can just choose any full exceptional collection in $\cD(X_{k-1}^n)$ and consider its pullback to~$X_k^n$ as the starting block.
It is elementary to check that it extends to a rectangular Lefschetz collection of length $n + 1$.
However, the $S_k$-symmetry in this construction is broken, and it cannot be performed in the equivariant category.

For $k = 1$ the Problem~\ref{problem:xkn} is trivial (the desired collection is just 
the Beilinson exceptional collection $\cO,\cO(1),\dots,\cO(n)$ of line bundles on $\mathbb{P}^n$).
Furthermore, for~\mbox{$k = 2$} the Problem~\ref{problem:xkn} was essentially solved in {\cite{Rennemo}}.

The main result of our paper is a partial solution to the Problem~\ref{problem:xkn}.

First, we construct in Theorem~\ref{theorem:semiorthogonality} a rectangular $S_k$-invariant Lefschetz exceptional collection in $\cD(X_k^n)$ 
whose cardinality in case of coprime $k$ and $n + 1$ equals the rank of the Grothendieck group of $X_k^n$ (by Elagin's Theorem, see Theorem~\ref{theorem:elagin}, 
this gives an exceptional collection in the equivariant category, whose length equals the rank of its Grothendieck group).
So, it is natural to expect that this collection is full and (in the coprime case) gives a solution to Problem~\ref{problem:xkn}.
However, in general we could not prove its fullness.

Our second main result is a proof of fullness of the above collection for $k = 3$ and~$n = 3p$ or $n = 3p + 1$ (this ensures that $k$ and $n + 1$ are coprime).

We also perform a first step in the direction of non-coprime $k$ and $n + 1$ by constructing a minimal $S_3$-invariant Lefschetz exceptional collection in $\cD(X_3^2)$ 
(including a proof of its fullness).

Besides that we also solve the Problem~\ref{problem:xkn} for $n = 1$, that is, construct a rectangular $S_k$-invariant Lefschetz collection of length $2$ in~$\cD(X_k^1)$ 
when $k$ is odd, and a minimal Lefschetz collection when $k$ is even.
However, this case is much more simple than the case $k=3$ discussed above.

An interesting feature of the Lefschetz collections that we construct in Theorem~\ref{theorem:semiorthogonality}
is that they resemble very much the minimal Lefschetz collections in the derived categories of the Grassmannians $\operatorname{\mathrm{Gr}}(k,n+1+k)$ constructed by Anton Fonarev, see \cite{Fonarev}.
It would be very interesting to understand the relations between these, 
since on one hand, this suggests a possible solution to the Problem~\ref{problem:xkn} for other values of $k$
(by considering analogues of Fonarev's collections),
and on the other hand, a solution to the Problem~\ref{problem:xkn} can help in dealing with the Grassmannians $\operatorname{\mathrm{Gr}}(k,n)$ when $k$ and $n$ are not coprime 
(in this case there is no rectangular collection on the Grassmannian, and a minimal collection is not quite known).

This paper is organized as follows. 
In Section~\ref{Preliminaries} we recall the definitions of full exceptional collections, Lefschetz and rectangular decompositions, and Elagin's Theorem.
In Section~\ref{A Lefschetz collection} we construct an $S_k$-invariant exceptional collection in $\cD(X_k^n)$ and discuss numerical restrictions
for the existence of a rectangular Lefschetz collection and some numerical bounds for a minimal Lefschetz collection.
Finally, in Section~\ref{section:fullness} we prove fullness of the constructed collections for $X_k^1$, $X_3^{3p}$, $X_3^{3p + 1}$ and $X_3^2$ respectively.

The author is grateful to A.~Kuznetsov for constant attention to this work.

\section {Preliminaries}\label{Preliminaries}

Given an algebraic variety $X$ we denote the bounded derived category $\cD^b(\coh (X))$ of coherent sheaves on $X$ by $\cD(X)$. 
In this paper we concentrate on the case when~$X$ is a power of a projective space
\begin{equation*}
X = X_k^n = (\mathbb{P}^n)^k,
\end{equation*}
In some cases, we will omit the indices $k$ and $n$ and write $\cD(X)$ instead $\cD(X_k^n)$.

\subsection{Exceptional collections in $\cD(X_k^n)$}

Clearly, $X_k^n$ is a smooth projective variety with $dim(X)=kn$.
Its Picard group is isomorphic to $\operatorname{\mathrm{Pic}}(X_k^n) \cong \mathbb{Z}^k$
and has a basis consisting of the pullbacks of hyperplane classes of the factors.
For $a = (a_1,\dots,a_k) \in \mathbb{Z}^k$ we write
\begin{equation*}
\cO(a) = \cO(a_1,\dots,a_k) = \cO(a_1) \boxtimes \dots \boxtimes \cO(a_k)
\end{equation*}
for the corresponding line bundle on $X_k^n$.
We note that by the K\"unneth formula
\begin{equation}
\label{eq:ext-kunneth}
\Ext^\bullet(\cO(a),\cO(b)) \cong \bigotimes_{i=1}^k \Ext^\bullet(\cO(a_i),\cO(b_i)).
\end{equation}
In particular, any line bundle on $X_k^n$ is exceptional, and the line bundles $\cO(a)$ and~$\cO(b)$ are semiorthogonal, i.e., $\Ext^\bullet(\cO(a),\cO(b))$ is equal to $0$, if and only if the pair~$(\cO(a_i),\cO(b_i))$
on $\mathbb{P}^n$ is semiorthogonal  for at least one $i$.
In view of Bott's formula for the cohomology of line bundles on a projective space, we can rewrite the semiorthogonality condition as 
\begin{equation}
\label{eq:semiorthogonality}
\Ext^\bullet(\cO(a),\cO(b)) = 0
\text{ if and only if $0 < a_i - b_i \le n$ for some $1 \le i \le k$.}
\end{equation}
This property allows to verify easily semiorthogonality of collections of line bundles.
For fullness, the following observations are useful.

For a subset $I \subset \{1,\dots,k\}$ of indices define the set $[0,n]^I \subset \operatorname{\mathrm{Pic}}(X_k^n)$ as 
\begin{equation*}
[0,n]^I=\left \{a\in \Z^k \mid a_i \in [0,n] \text{ if } i \in I\text{ and } a_i=0 \text{ if } i \notin I \right \}.
\end{equation*}
If $I=\{1,\dots,k\}$, then denote $[0,n]^I$ by $[0,n]^k$.

\begin{Th} 
\label{generate} 
The collection $\{ \cO(a) \}_{a \in [0,n]^k}$ \textup(lexicographically ordered\textup) is a full exceptional collection in $\cD(X_k^n)$.
\end{Th}
\begin{proof}
Semiorthogonality of the collection follows easily from~\eqref{eq:semiorthogonality}.
For fullness we refer to~\cite{Samokhin}.
\end{proof}

We will also need the following simple consequence of the fullness of the above collection.

\begin{Cor} 
\label{bbb2}
Let\/ $\mT$ be a triangulated subcategory of $\cD(X_k^n)$. 
Assume that for some subset $I\subset \{1,\dots,k\}$ and some $a \in \operatorname{\mathrm{Pic}}(X_k^n)$ one has $\cO(a + b) \in \mT$ for any~$b \in [0,n]^I$.
Then the same holds true for any $b \in \mathbb{Z}^I$.
\end{Cor}

\begin{proof}
First assume $a = 0$. 
Then the collection $\{ \cO(b) \}_{b \in [0,n]^I}$ is just the pullback of the full exceptional collection in 
\begin{equation*}
X_I^n = \prod_{i\in I}\mathbb{P}^n
\end{equation*}
with respect to the natural projection $X_k^n \to X_I^n$.
Consequently, by Theorem~\ref{generate} the category $\mT$ contains the pullback of any line bundle on $X_I^n$, 
and this is just the claim of the lemma in this case.

For arbitrary $a$ just note that $\{ \cO(a + b) \}_{b \in [0,n]^I}$ is the twist of $\{ \cO(b) \}_{b \in [0,n]^I}$ by~$\cO(a)$.
Since a line bundle twist is an autoequivalence of $\cD(X_k^n)$, the general claim follows.
\end{proof}

\subsection{Semiorthogonal and Lefschetz decompositions}

In some cases it is slightly more convenient to work with semiorthogonal decompositions than with exceptional collections.
Here, we remind the corresponding definitions.

\begin{Def} Suppose $\mB_0, \dots, \mB_d$
are full triangulated subcategories of $\mT$ such that $\Hom(\mB_i, \mB_j) = 0$ for all $i > j$. 
We say that $\mB_0, \dots, \mB_d$ form a \textit{semiorthogonal decomposition} of $\mT$ if the smallest full triangulated subcategory of $\mT$ containing $\mB_i$ for all $i$ coincides with $\mT$. 
\end{Def}

We will denote a semiorthogonal decomposition by 
\begin{equation*}
\langle \mB_0, \dots, \mB_d \rangle = \mT.
\end{equation*}

Assume that $\mT = \dbcoh$ and a line bundle $\mathcal{L}$ on $X$ is given.
For an object $F$ in~$\cD(X)$ we denote
\begin{equation*}
F(i) := F \otimes \mathcal{L}^i
\end{equation*}
the image of $F$ under the autoequivalence of $\mT$ given by the $\mathcal{L}^i$-twist, 
and for a subcategory $\mA \subset \mT$ we denote
\begin{equation*}
\mA(i) := \{ F(i) \mid F \in \mA \} \subset \mT
\end{equation*}
the image of $\mA$ under this autoequivalence.

A semiorthogonal decomposition 
\begin{equation}
\label{Lefschetz}
\cD(X) = \langle \mA_0, \mA_1(1), \dots, \mA_d(d) \rangle
\end{equation}
is called \textit{Lefschetz decomposition} if $\mA_{i+1} \subset \mA_{i}$ for all $0 \le i < d$.

We say that a Lefschetz decomposition~\eqref{Lefschetz} 
is \textit{rectangular} if $\mA_0 = \dots = \mA_{d}$. 
A rectangular decomposition can be simply written as
\begin{equation}
\label{rectangular}
\cD(X) = \langle \mA, \mA (1), \dots, \mA(d) \rangle, 
\end{equation}
where $\mA = \mA_0$.

\subsection{Exceptional collections in equivariant derived categories}

Assume a finite group $G$ acts on a smooth projective variety $X$.
The following result of Alexei Elagin gives a way to construct an exceptional collection in the equivariant derived category $\cD_G(X)$.

\begin{Th}[{\cite[Theorem~2.3]{Elagin}}]
\label{theorem:elagin}
Assume that $E_1,\dots,E_r$ is a full $G$-invariant exceptional collection in $\dbcoh$, that is, the $G$-action induces a permutation of objects of the collection. Assume $s$ is the number of $G$-orbits on $\{E_1, \dots, E_r\}$ and let~$E_{i_1}, \dots E_{i_s}$, $i_1 < \dots <i_s$ be their representatives. For each $1 \le t \le r$ let $H_t$ be the stabilizer of $E_{i_t}$ and assume that for each $t$ the object $E_{i_t}$ admits an~$H_t$-equivariant structure. Then there exists a full exceptional collection of the equivariant category
\begin{equation*}
\cD_G(X) = \langle \be_{i_1}^{(1)},\dots,\be_{i_1}^{(m_1)},\dots,\be_{i_s}^{(1)},\dots,\be_{i_s}^{(m_s)} \rangle.
\end{equation*}
Here $\be_{i_t}^{(j)} = E_{i_t} \otimes V_t^{(j)}$, where $V_t^{(1)}$, \dots, $V_t^{(m_t)}$ are all irreducible representations of~$H_t$ up to isomorphism, and we consider the natural $G$-equivariant structure on $\be_{i_t}^{(j)}$. 
\end{Th}

We note that any line bundle on $X = X_k^n$ has a natural equivariant structure with respect to the subgroup of $S_k$ that stabilizes it.
Indeed, for this it is enough to note that the line bundle $\cO(i,i,\dots,i)$ is $S_k$-equivariant for each $i$.
Thus, the above theorem applies to any exceptional collection formed by line bundles on $X_k^n$ as soon as it is $S_k$-invariant.
To ensure that the resulting collection in the equivariant category is Lefschetz we will use the following evident observation.

\begin{Cor}
Assume that $L$ is a $G$-equivariant line bundle on $X$ and $E_1,\dots,E_r$ is a Lefschetz exceptional collection with respect to $L$
which satisfies the assumptions of Theorem~\textup{\ref{theorem:elagin}}.
Then the corresponding exceptional collection in the equivariant category is also Lefschetz.
Moreover, if the original collection is rectangular then so is the equivariant one with the same number of blocks.
\end{Cor}
\begin{proof}
Let $E_1,\dots,E_{r_0}$ be the starting block of the original Lefschetz collection and~$s_0$ be the number of $G$-orbits in the block $E_1,\dots,E_{r_0}$.
Then it is straightforward to check that~$\be_{i_1}^{(1)},\dots,\be_{i_1}^{(m_1)},\dots,\be_{i_{s_0}}^{(1)},\dots,\be_{i_{s_0}}^{(m_{s_0})}$  can serve as
the starting block of a Lefschetz collection in $\cD_G(X)$.
From the equivariance of $L$ it is also clear that the property of being rectangular is preserved by this construction.
\end{proof}

Thus, to construct a (rectangular) Lefschetz collection in $\cD_{S_k}(X_k^n)$ it is enough to construct 
a (rectangular) $S_k$-invariant Lefschetz collection in $\dbcoh$ consisting of line bundles.
This is what we do in the next sections.

\section{A Lefschetz collection and numerical minimality}
\label{A Lefschetz collection}

In this section we construct a Lefschetz $S_k$-invariant exceptional collection on $X_k^n$
and find some numerical conditions for minimality of a Lefschetz exceptional collection.
In what follows we always denote
\begin{equation*}
h := n + 1.
\end{equation*}

\subsection{A Lefschetz collection}

We consider the following two $S_k$-invariant subsets of the lattice $\operatorname{\mathrm{Pic}}(X_k^n) = \mathbb{Z}^k$:
\begin{align}
\label{eq:enkr}
\Enkr  &= \left\{S_k \cdot (c_1, \dots, c_k) \mid c_1 \ge \dots \ge c_k = 0 \text{ and } kc_i \le h(k-i) \right\},
\intertext{and}
\label{eq:enk}
\Enk &= \left\{S_k \cdot (c_1, \dots, c_k) \mid c_1 \ge \dots \ge c_k = 0 \text{ and } kc_i < h(k-i) \text{ for } i \not= k \right\}.
\end{align}
Note that the only difference in the definitions of $\Enk$ and $\Enkr$ is that a non-strict inequality in~\eqref{eq:enkr} 
is replaced by a strict one in~\eqref{eq:enk}.
In particular, 
\begin{equation*}
\Enk \subset \Enkr, 
\end{equation*}
and if all the fractions $h(k-i)/k$ for $1 \le i \le k-1$ are non-integer, i.e., when $h$ and~$k$ are coprime, we have an equality $\Enk = \Enkr$.

We consider the above two sets with the lexicographical order restricted from $\Z^k$.

\begin{Lem} \label{excep}
The set of line bundles $\cO(c)$ for $c \in \Enkr$ is an exceptional $S_k$-invariant collection
with respect to the lexicographical order on $\Enkr$.
\end{Lem}
\begin{proof}
Follows from the evident inclusion $\Enkr \subset [0,n]^k$ and Theorem~\ref{generate}.
\end{proof}

Since the set $\Enk$ is an $S_k$-invariant subset in $\Enkr$, the collection of line bundles~$\cO(c)$ for $c \in \Enk$ is also an exceptional $S_k$-invariant collection
with respect to the lexicographical order on $\Enk$.

We denote by
\begin{equation}
\label{eq:ma-hma}
\mA = \langle \cO(c) \rangle_{c \in \Enk}
\qquad\text{and}\qquad
\widehat{\mA} = \langle \cO(c) \rangle_{c \in \Enkr}
\end{equation}
the subcategories in $\cD(X_k^n)$ generated by the above exceptional collections.
Furthermore, for each $c = (c_1,c_2,\dots,c_k) \in \mathbb{Z}^k$ we denote
\begin{equation*}
c(i) = (c_1 + i, c_2 + i, \dots, c_k + i), 
\end{equation*}
so that $\cO(c(i)) \cong \cO(c) \otimes \cO(i,i,\dots,i)$.

\begin{Th}
\label{theorem:semiorthogonality}
For any $h > i > j \ge 0$ we have $\Hom (\mA(i),\widehat{\mA}(j))=0$.

In particular, the category
\begin{equation}
\label{collection}
\mT:=\langle \widehat\mA, \mA(1) \dots \mA(n) \rangle \subset \cD(X_k^n)
\end{equation}
is generated by an $S_k$-invariant Lefschetz collection.
\end{Th}
\begin{proof}
Obviously, it is enough to prove the theorem for $j=0$, $i>0$.
In other words, it is enough to prove that for any $a \in \Enk$, $b \in \Enkr$ we have $\Hom(\cO(a(i)),\cO(b)) = 0$.
Furthermore, by $S_k$-invariance of the set $\Enk$, we can assume that $a_1 \ge \dots \ge a_k = 0$.

First, assume that $a_1 + i < h$.
Then $a_t + i < h$ for all $t$.
On the other hand, by definition of $\Enkr$ we have $b_t = 0$ for some $t$.
Then $0 < (a_t + i) - b_t \le n$, hence we have~$\Hom(\cO(a(i)),\cO(b)) = 0$ by~\eqref{eq:semiorthogonality}.

So, from now on we can assume that $a_1 + i \ge h$.
At the same time $a_k + i = i < h$.
Let $r$ be the maximal index such that
\begin{equation*}
a_r + i \ge h
\qquad \text{and} \qquad
a_{r+1} + i < h.
\end{equation*}
The first of these inequalities implies 
\begin{equation*}
i \ge h - a_r > h - h(k-r)/k = hr/k.
\end{equation*}
On the other hand, consider all $t$ such that
\begin{equation*}
b_t \le hr/k.
\end{equation*}
Note that by definition of~$\Enkr$ there are at least $r + 1$ such $t$ (corresponding to the smallest $r+1$ values of $b_t$),
hence for some of these we have $t \ge r + 1$.
For such $t$ we have 
\begin{equation*}
hr/k < i \le a_t + i \le a_{r+1} + i < h
\qquad \text{and} \qquad
0 \le b_t \le hr/k.
\end{equation*}
In particular, $0 < a_t + i - b_t < h$.
Hence $\Hom(\cO(a(i)),\cO(b)) = 0$ by~\eqref{eq:semiorthogonality}.
\end{proof}

Below we will prove that the category $\mT$ defined by \eqref{collection}  is equal to $\cD(X_k^n)$ in case 
$n \not= 2 \text{ mod } 3$,~$k = 3$ (Subection~\ref{3p}) and $n = 1$ and any $k$ (Subection~\ref{n=1}). In particular, for $\mathrm{gcd}(h,k)=1$ the right side of \eqref{collection} gives a rectangular Lefschetz decomposition of $\cD(X_k^n)$.

However, in general the sum of the ranks of the Grothendieck groups of the components of \eqref{collection} is less than the rank of the Grothendieck group of $\cD(X_k^n)$, so it requires a modification. In Subsection~\ref{x_3^2} we show how such  a modification can be performed for $n=2$ and $k=3$.

\subsection{Numerical restrictions} \label{Numerical_restrictions}

We keep the notation $h=n+1$ and let $V$ be a vector space of dimension $h$, so that~$\mathbb{P} (V) = \mathbb{P}^n$.
Denote by
\begin{equation*}
\K := K_0 (\mathbb P(V)) \otimes \mathbb C,
\end{equation*}
the complexified Grothendieck group of coherent sheaves on $\mathbb{P}(V)$.
It is also a vector space of dimension $h$.
Moreover, we have
\begin{equation*}
K_0(X_k^n) \otimes \mathbb{C} = K_0 (\mathbb P(V)^{k}) \otimes \mathbb C \cong \K^{\otimes k}.
\end{equation*}

The group $\operatorname{\mathrm{GL}}(\K)$ acts naturally on the vector space $\K^{ \otimes k}$,
and the group~$S_k$ acts on $\K^{ \otimes k}$ by permutation of factors (this action is induced by the action of $S_k$ on~$X_k^n$). 
These two actions commute, therefore $\K^{ \otimes k}$ is a~$(\operatorname{\mathrm{GL}}(\K), S_k)$-bimodule. 
In the next lemma we describe a decomposition of $\K^{ \otimes k}$ into a direct sum of irreducible representations, provided by the Schur--Weyl duality.

We denote by $\rho(h,k)$ the set of all Young diagrams of $k$ boxes with at most $h$ rows, 
by $\Sigma^{\lambda} \K$ the irreducible representation of $\operatorname{\mathrm{GL}}(\K)$ corresponding to the Young diagram~$\lambda$,
(it is also known as the Schur functor assoicated with $\lambda$), 
and by~$R_{\lambda^T}$ the irreducible representation of $S_k$ corresponding to the transposed Young diagram~$\lambda^T$.

\begin{Lem} [Schur--Weyl duality, \cite{Fulton Harris}] \label{decomp} 

There exists an isomorphism of~$\operatorname{\mathrm{GL}}(\K) \times S_k$ representations:
\begin{equation*}
\K^{ \otimes k} = \underset{\lambda \in \rho(h,k)}{\bigoplus} \Sigma^{\lambda} \K \otimes R_{\lambda^T}.
\end{equation*} 
In other words, the decomposition of $\K^{ \otimes k}$ into a direct sum of irreducible $S_k$-representations 
contains~$\dim (\Sigma^{\lambda} \K)$ copies of the irreducible representation $R_{\lambda^T}$.
\end{Lem}

The above decomposition allows to give a simple necessary condition for the existence of a rectangular $S_k$-invariant Lefschetz collection in $\cD(X_k^n)$.
In what follows we call it the \emph{divisibility criterion}.

\begin{Cor}
\label{div h} 
If a rectangular $S_k$-invariant Lefschetz decomposition of length $h$ of~$\cD(X_k^n)$ exists, then $h$ divides $\dim \Sigma^\lambda\K$ for all $\lambda \in \rho(h,k)$.
\end{Cor}
\begin{proof}
Assume $\cD(X_k^n) = \langle \mA_0, \mA_0(1), \dots, \mA_0(n) \rangle$ is a rectangular $S_k$-invariant Lefschetz decomposition.
Then we have
\begin{equation*}
\K^{ \otimes k} = K_0(X_k^n) \otimes \mathbb{C} = (K_0(\mA_0) \otimes \mathbb{C})^{\oplus h}.
\end{equation*}
Since $\mA_0$ is $S_k$-invariant, $K_0(\mA_0) \otimes \mathbb{C} \subset K_0(X_k^n) \otimes \mathbb{C}$ is an $S_k$-subrepresentation, 
so the above equality shows that the multiplicity of each irreducible summand of $\K^{ \otimes k}$ is divisible by $h$.
\end{proof}

The same argument as above gives the following bound for the ranks of the Grothendieck groups of components of an arbitrary $S_k$-invariant Lefschetz decomposition of $\cD(X_k^n)$.
Denote by $\lfloor t \rfloor$ and $\lceil t \rceil$ the lower and upper integral parts of~$t$.

\begin{Cor} \label{Lef}
Suppose $\cD(X_k^n) = \langle \mA_0, \mA_1(1), \dots, \mA_n(n) \rangle$ is a Lefschetz $S_k$-invariant decomposition. 
Let $r_i$ be the rank of $K_0(\mA_i)$. 
Then 
\begin{equation*}
r_0 \ge \underset{\lambda \in \rho(h,k)}{\sum} \left\lceil\frac {\dim \Sigma^{\lambda} \K} {h}\right\rceil \dim  R_{\lambda^T} 
\quad\text{and}\quad
r_n \le \underset{\lambda \in \rho(h,k)}{\sum} \left\lfloor\frac {\dim \Sigma^{\lambda} \K} {h}\right\rfloor \dim  R_{\lambda^T}.
\end{equation*}
\end{Cor}

\begin{proof}
Suppose that 
\begin{equation*}
K_0(\mA_i) \otimes \mathbb{C}=\underset{\lambda \in \rho(h,k)}{\sum} R_{\lambda^T}^{\oplus a^\lambda_i}.
\end{equation*}

From \ref{decomp} we get that $\underset{0\le i \le n}{\sum} a^\lambda_i=\dim \Sigma^{\lambda} \K$ for any $\lambda$.
Since $\mA_j \subset \mA_i$ for any~$i<j$, we have~$a^\lambda_j \le a^\lambda_i$ for any $\lambda$ and $i<j$. Thus
\begin{equation*}
  a_0^\lambda \ge \left\lceil\frac {\dim \Sigma^{\lambda} \K} {h}\right\rceil
\quad\text{and}\quad
a_n^\lambda \le \left\lfloor\frac {\dim \Sigma^{\lambda} \K} {h}\right\rfloor.  
\end{equation*}

This completes the proof. 
\end{proof}

As an example we consider the case $n=1$ and $k = 2m$.

\begin{Cor} \label{restriction for k=2m}
Suppose $\cD(X_{2m}^1)) = \langle \mA_0, \mA_1(1) \rangle$ is a Lefschetz $S_{2m}$-invariant decomposition. 
Let $r_i$ be the rank of $K_0(\mA_i)$. Then $r_0 -r_1 \ge \binom {2m} {m}$.
\end{Cor}

In Subsection \ref{n=1} we will show that the above inequality is sharp.

\begin{proof}
Any diagram in $\rho(2,2m)$ is of the shape
\begin{equation*}
\lambda(l) := (2m-l,l). 
\end{equation*}
for some $0 \le l \le m$. 
By Weyl dimension formula we have
\begin{equation*}
\dim (\Sigma^{\lambda(l)} \K) = 2m-2l+1,
\end{equation*}
and by the hook-length formula 
\begin{equation*}
\dim R_{\lambda(l)^T} = \frac{2m!\,(2m-2l+1)} {(2m-l+1)!\,l!} = \frac{2m - 2l + 1}{2m + 1} \binom{2m + 1}{l}.
\end{equation*}
For each $l$ we have $\lceil \frac{2m - 2l + 1}2 \rceil - \lfloor \frac{2m - 2l + 1}2 \rfloor = 1$, hence by Corollary~\ref{Lef}, we have
\begin{align*}
r_0 - r_1 \ge \sum_{l=0}^m  \dim  R_{\lambda(l)^T} 
&= \sum_{l=0}^m \frac{2m - 2l + 1}{2m + 1} \binom{2m + 1}{l} \\
&= \sum_{l=0}^m \binom{2m + 1}{l} - 2 \sum_{l=0}^m \binom{2m}{l - 1}.
\end{align*}
The first sum is equal to $2^{2m}$, and the second is equal to $2^{2m} - \binom{2m}{m}$, so we conclude that $r_0 - r_1 \ge \binom{2m}{m}$.
\end{proof}

If we restrict to the case of $S_k$-invariant Lefschetz collections, the inequalities of Corollary \ref{Lef} can be, in general, improved,
because in this case each $K_0(\mA_i)$ is a \emph{permutation representation} of $S_k$.

As an example, we consider the case $k = 3$, $n = 2$ (so that $h = 3$).
In this case
the set $\rho(3,3)$ consists of three Young diagrams: $(3)$, $(2,1)$, and $(1,1,1)$, and 
\begin{equation*}
\dim \Sigma^{(3)}\K = 10,
\qquad
\dim \Sigma^{(2,1)}\K = 8,
\qquad
\dim \Sigma^{(1,1,1)}\K = 1,
\end{equation*}
while
\begin{equation*}
\dim R_{(3)^T} = 1,
\qquad 
\dim R_{(2,1)^T} = 2,
\qquad 
\dim R_{(1,1,1)^T} = 1.
\end{equation*}
Consequently, if $\dbcoh = \langle \mA_0, \mA_1(1), \mA_2(2) \rangle$ is an $S_k$-invariant Lefschetz decomposition and $r_i$ is the rank of $K_0(\mA_i)$, then by Corollary~\ref{Lef} we have
\begin{align*}
r_0 &\ge \left\lceil\frac {10} {3}\right\rceil \cdot 1 + \left\lceil\frac {8} {3}\right\rceil \cdot 2 + \left\lceil\frac {1} {3}\right\rceil \cdot 1 
= 4+6+1=11\\
\intertext{and} 
r_2 &\le \left\lfloor\frac {10} {3}\right\rfloor \cdot 1 + \left\lfloor\frac {8} {3}\right\rfloor \cdot 2 + \left\lfloor\frac {1} {1}\right\rfloor \cdot 1 
= 3 + 4 + 0 = 7.
\end{align*}

On the other hand, we can prove the following result.

\begin{Prop} 
\label{proposition:x32-r0-lower-bound}
Assume $\langle \mA_0,\mA_1(1),\mA_2(2)\rangle=\cD(X_3^2)$ is a Lefschetz decomposition, 
such that each component $\mA_i$ is generated by an $S_3$-invariant exceptional collection~$\{E_{i,j}\}_{j=1}^{r_i}$. Then $r_0 \ge 13$ and $r_2 \le 7$.
\end{Prop}
\begin{proof}
The classes of exceptional objects $E_{i,j}$ form a basis of the Grothendieck group~$K_0(\mA_i)$.
Since the collection is $S_3$-invariant, this basis is permuted by the group action, i.e., $K_0(\mA_i) \otimes \C$ is a sum of permutation representations.
There are three such representations:
\begin{equation*}
\C[S_3] \cong R_{(3)^T} \oplus R_{(2,1)^T}^{\oplus 2} \oplus R_{(1,1,1)^T},
\quad 
\C[S_3/S_2] \cong R_{(3)^T} \oplus R_{(2,1)^T},
\quad 
\C[S_3/S_3] \cong R_{(3)^T}.
\end{equation*}
Note that $R_{(1,1,1)^T}$ only appears as a summand of $\C[S_3]$.

On the other hand, by Lemma~\ref{decomp} we have
\begin{equation*}
K_0(\mA_0) \oplus K_0(\mA_1) \oplus K_0(\mA_2) \cong K_0(\cD(X_3^2)) \cong R_{(3)^T}^{\oplus 10} \oplus R_{(2,1)^T}^{\oplus 8} \oplus R_{(1,1,1)^T}.
\end{equation*}
Finally, by the Lefschetz property, we have $K_0(\mA_2) \subset K_0(\mA_1) \subset K_0(\mA_0)$.
This means that $R_{(1,1,1)}$ has to be a direct summand of $K_0(\mA_0)$, 
hence $K_0(\mA_0)$ contains the entire regular representation $\C[S_3]$, and implies $r_0 \ge r_1 + 6 \ge r_2 + 6$.
Therefore
\begin{equation*}
3r_0 \ge r_0 + (r_1+6) +(r_2 +6) 
= 27+6+6=39,
\end{equation*}
and hence $r_0 \ge 13$.

Since 
\begin{equation*}
2 r_2 \le r_1+r_2 = 27-r_0 \le 27-13= 14,
\end{equation*}
we have $r_2 \le 7$.
\end{proof}

In Section \ref{x_3^2} we will construct a full $S_3$-invariant Lefschetz exceptional  collection in $\cD(X_3^2)$ with $(r_0,r_1,r_2)=(13,7,7)$.

\subsection{Verifications of divisibility}

To check divisibility of the dimensions of $\Sigma^\lambda\K$ the following corollary of Littlewood--Richardson rule is useful.

\begin{Lem} [\cite{Fulton Harris}] \label{L-R}
 Let $\mu = (\mu_1, \mu_2, \dots, \mu_m)$ be a 
Young diagram. Then
\begin{equation*}
\Lambda^{\mu_1}\K \otimes \Lambda^{\mu_2}\K \otimes \dots \otimes \Lambda^{\mu_m}\K \cong 
\Sigma^{\mu^T} \K \oplus \left(\underset{\lambda < {\mu^T}}{\bigoplus} ({\Sigma^{\lambda} \K})^{{\oplus}c(\lambda, \mu)}\right),
\end{equation*}
where $<$ stands for the \textit{dominance order} \textup{\cite[Section~2.2]{Fulton}}, and $c(\lambda,\mu)$ are nonnegative integers.
\end{Lem}

The next proposition gives some necessary and sufficient conditions for divisibility.

\begin{Prop} \label{result} 
$(1)$ If $h$ divides $k$, then $\dim \Sigma^\lambda\K$ is not divisible by $h$ for some Young diagram $\lambda \in \rho(h,k)$.

\noindent$(2)$ If $k$ is not divisible by $h$ and for any integer  $r$ such that $1 \le r \le \mathrm{min}(k, h-1)$ 
the binomial coefficient~$\binom {h}{r}$ is divisible by $h$, then $\dim (\Sigma^{\lambda} \K)$ is divisible by $h$ for any Young diagram $\lambda \in \rho(h,k)$.
\end{Prop}

\begin{proof}
(1)
Suppose $k = ht$. 
Consider the Young diagram $\xi$ with $t$ columns of height~$h$. 
Then $\Sigma^\xi \K \cong (\det \K)^{\otimes t}$, hence $\dim (\Sigma^{\xi} \K) = 1$. 
Since $h = n+1 \ge 2$, we see that~$\dim (\Sigma^{\xi} \K)$ is not divisible by $h$.

(2) We use ascending induction on Young diagrams in $\rho(h,k)$ with respect to the dominance order.

{\it{Base.}} Suppose $k = h t + r$, where $t\in  \mathbb {Z}_{\ge 0}$ and $1 \le r \le h - 1$. 
It is clear that the smallest diagram $\omega \in \rho(h,k)$ is the diagram with $t$ columns of height $h$ and one column of height $r$. 
Then $\Sigma^\omega \K \cong (\det \K)^{\otimes t} \otimes \Lambda^r \K$, hence $\dim (\Sigma^{\omega} \K) = \binom {h}{r}$,
which is divisible by $h$ by the assumption of the proposition.

{\it{Induction step.}} Consider a diagram $\mu$ such that $\mu^T\in \rho(h,k)$. 
Suppose that for any $\lambda<\mu^T$, $\lambda \in \rho(h,k)$, the dimension~$\dim (\Sigma^{\lambda} \K)$ is divisible by $h$. 
Let us prove that $\dim (\Sigma^{\mu^T} \K)$ is also divisible by $h$. 
Using Lemma~\ref{L-R}, we get

\begin{multline*}
\dim (\Sigma^{\mu^T} \K)= \prod\limits_{i=1}^m {\dim (\Lambda^{\mu_i}\K)} - \dim \left(\underset{\lambda < {\mu^T}}{\bigoplus} ({\Sigma^{\lambda} \K})^{{\oplus}c(\lambda)} \right) = \\ 
= \prod\limits_{i=1}^m \textstyle \binom {h}{\mu_i} - \sum\limits_{\lambda < {\mu^T}} c(\lambda, \mu) \cdot \dim ({\Sigma^{\lambda} \K}).
\end{multline*}

By induction hypothesis, $\sum\limits_{\lambda < {\mu^T}} c(\lambda, \mu) \cdot \dim ({\Sigma^{\lambda} \K})$ is divisible by~$h$. 
Since $k$ is not divisible by $h$ we see that there exist $i$ such that $1 \le \mu_i \le h-1$. 
Clearly, $\mu_i \le k$. 
Therefore,~$1 \le \mu_i \le \min(k, h-1)$. 
Thus $\binom {h} {\mu_i}$ is divisible by $h$ by the assumption of the theorem. 
Hence $\prod\limits_{i=1}^m \textstyle \binom {h}{\mu_i}$ 
and consequently
$\dim (\Sigma^{\mu^T} \K)$ is divisible by~$h$.
\end{proof}

Note that to prove the inductive step we need only one $\binom {h} {\mu_i}$ to be divisible by $h$ for each $\mu=(\mu_1, \mu_2, \dots, \mu_m)$ with $\mu^T\in \rho(h,k)$. 
This suggests that the assumption of Theorem~\ref{result}(2) can be weakened.

Next, we discuss consequences of the above in the case $k = 3$.

\begin{Prop} \label{k=3} 
If $n = 3p + 2$, then the category $\cD(X_3^n)$ does not have a rectangular $S_3$-invariant Lefschetz decomposition of length $n + 1$.
\end{Prop}

\begin{proof} If $n=3p+2$, then $n>1$ and $h=n+1 \ge 3$. 
Thus for $k=3$ all Young diagrams of three boxes are in $\rho(h,k)$. 
These diagrams are $(1,1,1)$, $(2,1)$, and $(3)$.
The dimensions of the corresponding Schur functors are given by 
\begin{align*}
\dim \Sigma^{(3)}\K &= \textstyle\frac{(h+2)(h+1)h}{6},\\
\dim \Sigma^{(2,1)}\K &= \textstyle\frac{(h+1)h(h-1)}{3},\\
\dim \Sigma^{(1,1,1)}\K &= \textstyle\frac{h(h-1)(h-2)}{6}.
\end{align*}
(see for instance the dimension formula from~\cite[Exercise~6.4]{Fulton Harris}).

Thus, a necessary condition for the existence of a rectangular $S_3$-invariant Lefschetz decomposition of length $h$ is that the three numbers above
are divisible by $h$. 
This is equivalent to the integrality of the fractions 
\begin{equation*}
\frac {(h-1)(h-2)}{6},
\qquad
\frac {(h+1)(h-1)}{3},
\quad\text{and}\quad
\frac {(h+2)(h+1)}{6}.
\end{equation*}
It is easy to see that this condition holds if and only if $h$ is not divisible by $3$. 
Since~$h=n+1$ we obtain that this condition holds if and only if $n \not=3p+2$.
\end{proof}

In other words, we can expect the existence of the desired rectangular decomposition only if $n=3p$ or $n=3p+1$. In the Subsection \ref{3p} we prove that the desired rectangular decomposition exists in these cases.

\section{Fullness}
\label{section:fullness}

In this section we prove that the $S_k$-invariant Lefschetz collection~\eqref{collection} generates the category $\cD(X_k^n)$ when $n = 1$ and any $k$ (Subection \ref{n=1}) or $k = 3$  and $n \not= 2 \text{ mod } 3$ (Subsection \ref{3p})
and moreover provides a minimal $S_k$-invariant Lefschetz collection in it. We also discuss the case $k=3$, $n=2$ (Subsection \ref{x_3^2}) that shows that in general collection~\eqref{collection} needs a modification.

\subsection{Minimal Lefschetz decomposition for $\cD(X_k^1)$}
\label{n=1}

First, we consider the case $n=1$. 
Recall the definition~\eqref{eq:ma-hma} of $S_k$-invariant subcategories $\mA \subset \widehat\mA \subset \cD(X_k^1)$.
In the case $n = 1$ it can be rewritten as
\begin{align}\label{eq:ma-hma-n1}
\mA &= \{ a \in [0,1]^k \mid \operatorname{\mathrm{Card}} \left\{i \mid a_i = 0 \right\}> k/2 \},\\
\widehat\mA &= \{ a \in [0,1]^k \mid \operatorname{\mathrm{Card}} \left\{i \mid a_i = 0 \right\} \ge k/2 \},
\end{align}

where $\operatorname{\mathrm{Card}}$ stands for the cardinality of a set.
If $k$ is odd, $\mA=\wa$.

\begin{Th}\label{th:n=1} 
We have $S_k$-invariant Lefschetz decompositions
\begin{equation*}
\cD(X_k^1) = 
\begin{cases}
\langle \mA, \mA(1) \rangle, & \text{if $k = 2m + 1$},\\
\langle \widehat\mA, \mA(1) \rangle, & \text{if $k = 2m$}.\\
\end{cases}
\end{equation*}
Moreover, these are minimal Lefschetz collections.
\end{Th}

\begin{proof}
By definition both subcategories $\mA$ and $\widehat\mA$ are generated by $S_k$-invariant exceptional collections. Moreover, by Theorem \ref{theorem:semiorthogonality} they are semiorthogonal. Thus for the first part of the theorem it is enough to show that $\widehat\mA$ and $\mA(1)$ generate $\cD(X_k^1)$.
For this we show that
\begin{equation}
\label{eq:oa-in-a1}
\cO(b) \in \mA(1)
\qquad
\text{if $b \in [0,2]^k$ and $\operatorname{\mathrm{Card}} \left\{i \mid b_i=1 \right\} \ge m+1$}.
\end{equation} 

Indeed, by definition of $\mA(1)$ we have 

\begin{equation}
\label{eq:extra}
\cO(b) \in \mA(1)
\qquad
\text{if $b \in [1,2]^k$ and $\operatorname{\mathrm{Card}} \left\{i \mid b_i=1 \right\} \ge m+1$}.
\end{equation} 

Note that $\cO(1, \dots, 1) \in \mA(1)$. 
We apply Corollary~\ref{bbb2} to $a=(1, \dots, 1)$ and any~$I$ of cardinality $m$. It proves that for any $b\in [0,2]^k$ such that $b_i=1$ for $i \notin I$ we have~$\cO(b) \in \mA(1)$. This proves \eqref{eq:oa-in-a1}.

Combining ~\eqref{eq:oa-in-a1} with the definition of $\widehat\mA$, 
we deduce that all line bundles~$\cO(a)$ with $a \in [0,1]^k$ 
are contained in the subcategory of $\cD(X_k^1)$ generated by~$\widehat\mA$ and $\mA(1)$.
By Theorem~\ref{generate} this proves the first part of Theorem \ref{th:n=1}.

It remains to show the minimality of the constructed Lefschetz collection.
For odd~$k$ the collection is rectangular of length $d=2$, hence minimal (see \cite[Subsection~2.1]{KuznetsovSmirnov}), so there is nothing to prove.
For even $k$ we note that the ranks of the Grothendieck groups of $\widehat\mA$ and $\mA$ are given by
\begin{equation*}
r_0 = 2^{2m} + \frac{1}{2}\binom {2m} {m}
\qquad\text{and}\qquad
r_1=2^{2m}-\frac{1}{2}\binom {2m} {m}
\end{equation*}
respectively.
In particular, $r_0 - r_1 = \binom{2m}{m}$, hence the collection is minimal by Corollary \ref{restriction for k=2m}.
\end{proof}

\subsection{Lefschetz decompositions for $\cD(X_3^{3p})$ and~$\cD(X_3^{3p+1})$} 
\label{3p}

In this subsection we prove the following

\begin{Th} 
\label{theorem:x3-3p}
Let $n = 3p$ or $n=3p+1$. The categories $\mA$ defined by~\eqref{eq:ma-hma} and~\eqref{eq:enk} generate an $S_3$-invariant rectangular Lefschetz collection 
\begin{equation*}
\cD(X_3^n) = \langle \mA, \mA(1), \dots, \mA(n) \rangle. 
\end{equation*}
\end{Th}

The proof takes the rest of the section.
As in the case of Theorem~\ref{theorem:x32} we denote by~$\mT$ the triangulated subcategory of $\cD(X)$ generated by the above Lefschetz collection. Note that $\mT$ is $S_3$-invariant.
By subsequent applications of Corollary~\ref{bbb2} we will show that many other line bundles are contained in $\mT$, until in the end we have~$\cO(a) \in \mT$ 
for all $a \in [0,n]^3$ and conclude by Theorem~\ref{generate}.

We will prove the statement of Theorem \ref{theorem:x3-3p} for $n = 3p$ and $n=3p+1$ in parallel. Denote by $T$ the set of all $a \in \Z^3$ such that $\cO(a) \in \mT$. Note that $T$ is $S_3$-invariant.

\begin{Prop} 
\label{one} 
For each $i \in [n-p,n]$ and $a \in \Z^3$ with $a_3 = i$, we have~$a \in T$.
\end{Prop}
\begin{proof}
Let us fix $i \in [n-p,n]$.
Consider a plane and mark on it all integral points~$(a_1,a_2)$ such that $(a_1,a_2, i) \in T$.
By definition~\eqref{eq:enk} all integral points of the polygon in Figure~\ref{figure:one} are marked.
The coordinates of its vertices $x_1,\dots,x_{12}$ are listed in the table below.

\begin{figure}[H]

      \caption{Illustration  for Step 1 of Proposition~\ref{one}.}\label{figure:one}
        \centering
\begin{tikzpicture} [scale=0.8]
\path [fill=black!30!white] (2,0) -- (4,0) -- (4,6) -- (2,6)--(2,0);
\draw [<->] (0,6) -- (0,0) -- (6,0);

 \draw[fill] (5,3) circle (0.7pt);
 \draw[fill] (4,4) circle (0.7pt);
 \draw[fill] (5,4) circle (0.7pt);
  \draw[fill] (3,2) circle (0.7pt);
 \draw[fill] (4,2) circle (0.7pt);
   \draw[fill] (2,1) circle (0.7pt);
 \draw[fill] (1,1) circle (0.7pt);
  \draw[fill] (3,5) circle (0.7pt);
 \draw[fill] (4,4) circle (0.7pt);
 \draw[fill] (4,5) circle (0.7pt);
  \draw[fill] (2,3) circle (0.7pt);
 \draw[fill] (2,4) circle (0.7pt);
   \draw[fill] (1,2) circle (0.7pt);

 \draw (3,2) --(4,2);
  \draw (4,2) --(5,3);
  \draw (5,4) --(4,4);
 \draw (5,4) --(5,3);
  \draw (1,1) --(2,1);
 \draw (2,1) --(3,2);
 \draw (2,3) --(2,4);
  \draw (2,4) --(3,5);
  \draw (4,5) --(4,4);
 \draw (4,5) --(3,5);
  \draw (1,1) --(1,2);
 \draw (1,2) --(2,3);

  \node [below] at (1,1) {$x_1$};
  \node [above, xshift=-0.5cm, yshift=0.2cm] at (1,2) {$x_2$};
  \node [above left] at (2,3) {$x_3$};
  \node [above left] at (2,4){$x_4$};
  \node [above left] at (3,5){$x_5$};
  \node [above] at (4,5) {$x_6$};
  \node [above right] at (4,4) {$x_7$};
  \node [right] at (5,4) {$x_8$};
  \node [right] at (5,3) {$x_9$};
  \node [right] at (4,2){$x_{10}$};
  \node [below right] at (3,2) {$x_{11}$};
  \node [below right] at (2,1){$x_{12}$};

\node [below] at (6,0) {$a_1$};
\node [left] at (0,6) {$a_2$};

 \draw[fill] (2.5,1.5) circle (0.7pt);
  \draw[fill] (2.5,4.5) circle (0.7pt);
 \draw [dashed] (2.5,1.5) --(2.5,4.5);
 \node [below right ] at (2.5,1.5) {\scriptsize$(n-p-i+c,n-2p-i+c)$};
  \draw[fill] (3.5,2) circle (0.7pt);
  \draw[fill] (3.5,5) circle (0.7pt);
  \draw [dashed] (3.5,2) --(3.5,5);
\end{tikzpicture}
\end{figure}
\begin{equation*}
\begin{array}{||c|c|c||}
\hline
  & n=3p & n=3p+1\\
\hline
x_1 & (i-2p,i-2p) & (i-2p-1,i-2p-1)\\
x_2 & (i-2p,i-p) & (i-2p-1,i-p-1)\\
x_3 & (i-p,i) & (i-p,i)\\
x_4 & (i-p,i+p) & (i-p,i+p+1)\\
x_5 & (i,i+2p) & (i,i+2p+1)\\
x_6 & (i+p,i+2p) & (i+p,i+2p+1)\\
x_7 & (i+p,i+p) & (i+p,i+p)\\
x_8 & (i+2p,i+p) & (i+2p+1,i+p)\\
x_9 & (i+2p,i) & (i+2p+1,i)\\
x_{10} & (i+p,i-p) & (i+p+1,i-p)\\
x_{11} & (i,i-p) & (i,i-p)\\
x_{12} & (i-p,i-2p) & (i-p-1,i-2p-1)\\
\hline
\end{array}\\
\end{equation*}

Our goal is to show that all integral points of the plane are in $T$. We do this in several steps.

{\bf Step 1.} 
For each $c \in[0,p]$ we apply Corollary~\ref{bbb2} with $a$ any integral point on the union of the edges~$[x_{10},x_{11}]$ and~$[x_{11},x_{12}]$ of the polygon in Figure~\ref{figure:one}, i.e., with~$a = (i+c,i-p,i)$, $I = \{2\}$ or~$a=(i+p+c,i-2p+c,i)$, $I = \{2\}$. 
Each dashed segment in Figure~\ref{figure:one} contains~$n$ integral points corresponding to line bundles contained in $\mT$.  

By Corollary~\ref{bbb2} we conclude that all points $(i+c,t,i)$, $(i+p+c,i-2p+c,i)$  are in $T$ for any~$t \in \Z$. In other words, all points in the grey vertical stripe in Figure~$1$ are in $T$.

{\bf Step 2.} Using $S_3$-symmetry of $T$ we conclude that all points in the horizontal grey stripe on Figure $2$ are in $T$.

{\bf Step 3.} Combining the results of Step 1 and Step 2 above, we see that $a \in T$ for any $a$ such that $(a_1,a_2) \in [i+p-n,i+p]^2$, $a_3=i$. In other words, all points in the square with vertices  $x_1$, $y_1$, $x_7$, $y_2$ in Figure $2$ are in $T$. Therefore we can apply Corollary~\ref{bbb2} with $a = (i+p-n,i+p-n,i)$ and $I = \{1,2\}$. We conclude that if~$a_3=i$, then~$a \in T$.

\begin{figure}[H]
  \caption{Illustration for Steps 2--3 of Proposition \ref{one}.}
  \centering
\begin{tikzpicture} [scale=0.8]

\path [fill=black!30!white] (2,0) -- (2,6) -- (4,6) -- (4,0) -- (2,0);
\path [fill=black!30!white] (0,2) -- (6,2) -- (6,4) -- (0,4) -- (0,2);
\draw [<->] (0,6) -- (0,0) -- (6,0);

   \draw (1,4) --(4,4);
   \draw (4,4) --(4,1);
   \draw (4,1) --(1,1);
   \draw (1,1) --(1,4);

 \draw[fill] (4,4) circle (0.7pt);
  \draw[fill] (1,1) circle (0.7pt);

 \draw (3,2) --(4,2);
  \draw (4,2) --(5,3);
  \draw (5,4) --(4,4);
 \draw (5,4) --(5,3);
  \draw (1,1) --(2,1);
 \draw (2,1) --(3,2);
 \draw (2,3) --(2,4);
  \draw (2,4) --(3,5);
  \draw (4,5) --(4,4);
 \draw (4,5) --(3,5);
  \draw (1,1) --(1,2);
 \draw (1,2) --(2,3);

\node [above right] at (4,4) {$x_7$};
\node [below] at (1,1) {$x_1$};
\node [right] at (4,1) {$y_2$};
\node [above] at (1,4) {$y_1$};

\node [below] at (6,0) {$a_1$};
\node [left] at (0,6) {$a_2$};

\end{tikzpicture}
\end{figure}
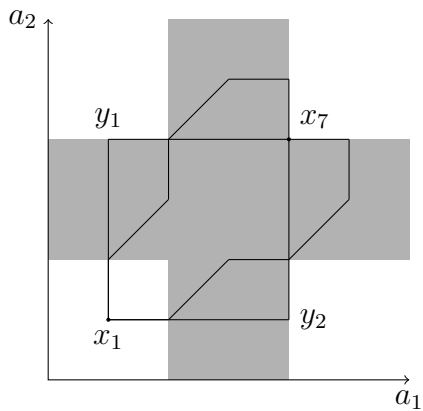

This completes the proof of Proposition \ref{one}.
\end{proof}

\begin{Prop} \label{two} For any $i \in [p,n-p-1]$, $a\in \Z^3$ such that $a_3=i$, we have~$a \in T$.
\end{Prop}
\begin{proof}
Let us fix $i \in [p,n-p-1]$.
Consider a plane and mark on it all integral points $(a_1,a_2)$ such that $\cO(a_1,a_2,i) \in \mT$.
By definition of~\eqref{eq:enk} all integral points of the polygon in Figure $3$ are marked. The coordinates of its vertices $x_1,\dots,x_{12}$ are listed in the table below.

\begin{figure}[H] \label{figure:three}
  \caption{Illustration for Step 1 of Proposition \ref{two}.}
  \centering
 \begin{tikzpicture} [scale=0.9]

\draw [<->] (1.5,6) -- (1.5,1.5) -- (6,1.5);
\path [fill=black!30!white] (2.5,1.5) -- (4,1.5) -- (4,6) -- (2.5,6)--(2.5,1.5);

 \draw[fill] (2.75,1.75) circle (0.7pt);
  \draw[fill] (2.75,4.75) circle (0.7pt);
 \draw [dashed] (2.75,1.75) --(2.75,4.75);
   \draw[fill] (3.5,2) circle (0.7pt);
  \draw[fill] (3.5,5) circle (0.7pt);
  \draw [dashed] (3.5,2) --(3.5,5);

 \draw[fill] (5,3) circle (0.7pt);
 \draw[fill] (4,4) circle (0.7pt);
 \draw[fill] (5,4) circle (0.7pt);
  \draw[fill] (3,2) circle (0.7pt);
 \draw[fill] (4,2) circle (0.7pt);
   \draw[fill] (2.5,1.5) circle (0.7pt);
 \draw[fill] (1.5,1.5) circle (0.7pt);

  \draw[fill] (3,5) circle (0.7pt);
 \draw[fill] (4,4) circle (0.7pt);
 \draw[fill] (4,5) circle (0.7pt);
  \draw[fill] (2,3) circle (0.7pt);
 \draw[fill] (2,4) circle (0.7pt);
   \draw[fill] (1.5,2.5) circle (0.7pt);

 \draw (3,2) --(4,2);
  \draw (4,2) --(5,3);
  \draw (5,4) --(4,4);
 \draw (5,4) --(5,3);
  \draw (1.5,1.5) --(2.5,1.5);
 \draw (2.5,1.5) --(3,2);
 \draw (2,3) --(2,4);
  \draw (2,4) --(3,5);
  \draw (4,5) --(4,4);
 \draw (4,5) --(3,5);
  \draw (1.5,1.5) --(1.5,2.5);
 \draw (1.5,2.5) --(2,3);

\node [below] at (1.5,1.5) {$x_1$};
\node [above, xshift=-0.5cm] at (1.5,2.5){$x_2$};
\node [below right] at (2,3) {$x_3$};
\node [below right] at (2,4) {$x_4$};
\node [above left] at (3,5){$x_5$};
 \node [above right] at (4,5) {$x_6$};
\node [above right] at (4,4) {$x_7$};
 \node [right] at (5,4) {$x_8$};
\node [right] at (5,3) {$x_9$};
\node [right] at (4,2) {$x_{10}$};
\node [below right] at (3,2) {$x_{11}$};
\node [below] at (2.5,1.5) {$x_{12}$};

\node [below] at (6,1.5) {$a_1$};
\node [left] at (1.5,6) {$a_2$};
\end{tikzpicture}
\end{figure}
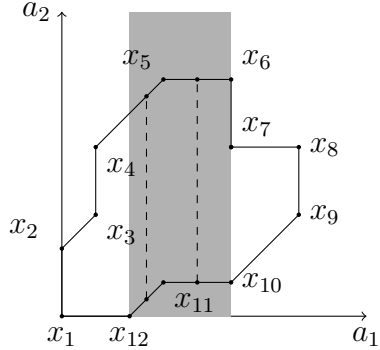

\begin{equation*}
\begin{array}{||c|c|c||}
\hline
  & n=3p & n=3p+1\\
\hline
x_1 & (0,0) & (0,0)\\
x_2 & (0,p) & (0,p)\\
x_3 & (i-p,i) & (i-p,i)\\
x_4 & (i-p,i+p) & (i-p,i+p+1)\\
x_5 & (i,i+2p) & (i,i+p+1)\\
x_6 & (i+p,i+2p) & (i+p,i+2p+1)\\
x_7 & (i+p,i+p) & (i+p,i+p)\\
x_8 & (i+2p,i+p) & (i+2p+1,i+p)\\
x_9 & (i+2p,i) & (i+2p+1,i)\\
x_{10} & (i+p,i-p) & (i+p+1,i-p)\\
x_{11} & (i,i-p) & (i,i-p)\\
x_{12} & (p,0) & (p,0)\\
\hline
\end{array}\\
\end{equation*}

Our goal is to show that all integral points of the plane are in $T$. We do this in several steps.

{\bf Step 1.} For any $c \in[0,i-p]$ we apply Corollary~\ref{bbb2} with $a$ any integral point on the union of the edges $[x_{10},x_{11}]$ and $[x_{11},x_{12}]$ of the polygon in Figure $3$, i.e., with~$a = (i+c,i-p,i)$, $I = \{2\}$ or~$a = (p+c,c,n-i)$, $I = \{2\}$. Each dashed segment in Figure $3$ contains $n$ integral points corresponding to line bundles contained in $\mT$. 

By Corollary~\ref{bbb2} we conclude that all points $(t,p+c,i)$, $(p+c,t,i)$ are in $T$ for any $t\in \Z$.  In other words, all points in the grey vertical stripe in Figure $3$ are in $T$.

{\bf Step 2.} Using $S_3$-symmetry of $T$ we conclude that  all points in the horizontal grey stripe on Figure $4$ are in $T$.

Combining the results of Step 1 and Step 2 above, we see that $a \in T$ for any $a$ such that $(a_1,a_2) \in [0,i+p]^2, a_3=i$. Since $i \in [p,n-p-1]$,  we have $i+p \ge 2p$. 

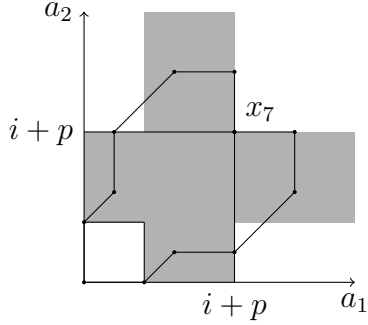
\begin{figure}[H]
  \caption{Illustration for Step 2 of Proposition \ref{two}.}
  \centering
 \begin{tikzpicture} [scale=0.8]

\path [fill=black!30!white] (1.5,2.5) -- (6,2.5) -- (6,4) -- (1.5,4) -- (1.5,2.5);

\path [fill=black!30!white] (2.5,1.5) -- (2.5,6) -- (4,6) -- (4,1.5) -- (2.5,1.5);

\draw [<->] (1.5,6) -- (1.5,1.5) -- (6,1.5);

 \draw[fill] (5,3) circle (0.7pt);
 \draw[fill] (4,4) circle (0.7pt);
 \draw[fill] (5,4) circle (0.7pt);
  \draw[fill] (3,2) circle (0.7pt);
 \draw[fill] (4,2) circle (0.7pt);
   \draw[fill] (2.5,1.5) circle (0.7pt);
 \draw[fill] (1.5,1.5) circle (0.7pt);

  \draw[fill] (3,5) circle (0.7pt);
 \draw[fill] (4,4) circle (0.7pt);
 \draw[fill] (4,5) circle (0.7pt);
  \draw[fill] (2,3) circle (0.7pt);
 \draw[fill] (2,4) circle (0.7pt);
   \draw[fill] (1.5,2.5) circle (0.7pt);

 \draw (3,2) --(4,2);
  \draw (4,2) --(5,3);
  \draw (5,4) --(4,4);
 \draw (5,4) --(5,3);
  \draw (1.5,1.5) --(2.5,1.5);
 \draw (2.5,1.5) --(3,2);
 \draw (2,3) --(2,4);
  \draw (2,4) --(3,5);
  \draw (4,5) --(4,4);
 \draw (4,5) --(3,5);
  \draw (1.5,1.5) --(1.5,2.5);
 \draw (1.5,2.5) --(2,3);

   \draw (1.5,4) --(4,4);
   \draw (4,4) --(4,1.5);
   \draw (1.5,2.5) --(2.5,2.5);
   \draw (2.5,1.5) --(2.5,2.5);

\node [above right] at (4,4) {$x_7$};
\node [below] at (4,1.5) {$i+p$};
\node [left] at (1.5,4) {$i+p$};
\node [below] at (6,1.5) {$a_1$};
\node [left] at (1.5,6) {$a_2$};
\end{tikzpicture}
\end{figure}

{\bf Step 3.} Note that by Proposition \ref{one} and $S_3$-symmetry of $T$ we have $a \in T$ if~$a_1 \in [n-p,n]$ or $a_2 \in [n-p,n]$. Using Step $2$ and the inequality $n-p \le 2p+1$ we get that $a \in T$ for any $a$ such that $(a_1,a_2) \in [0,n]^2, a_3=i$. Therefore we can apply Corollary~\ref{bbb2} with $a = (0,0,i)$ and $I = \{1,2\}$. We conclude that if $a_3=i$, then~$a \in T$.

\begin{figure}[H]
  \caption{Illustration for Step 3 of Proposition \ref{two}.}
  \centering
 \begin{tikzpicture} [scale=0.8]
\draw [<->] (1.5,6) -- (1.5,1.5) -- (6,1.5);

\path [fill=black!15!white] (1.5,1.5) -- (1.5,3.5) -- (3.5,3.5) -- (3.5,1.5) -- (1.5,1.5);
\path [fill=black!30!white] (1.5,3.5) -- (1.5,4.5) -- (4.5,4.5) -- (4.5,1.5) -- (3.5,1.5)-- (3.5,3.5) -- (1.5,3.5);

 \draw (1.5,3.5) --(3.5,3.5);
   \draw (3.5,3.5) --(3.5,1.5);
   \draw (1.5,4.5) --(4.5,4.5);
   \draw (4.5,1.5) --(4.5,4.5);

 \draw[fill] (5,3) circle (0.7pt);
 \draw[fill] (4,4) circle (0.7pt);
 \draw[fill] (5,4) circle (0.7pt);
  \draw[fill] (3,2) circle (0.7pt);
 \draw[fill] (4,2) circle (0.7pt);
   \draw[fill] (2.5,1.5) circle (0.7pt);
 \draw[fill] (1.5,1.5) circle (0.7pt);

  \draw[fill] (3,5) circle (0.7pt);
 \draw[fill] (4,4) circle (0.7pt);
 \draw[fill] (4,5) circle (0.7pt);
  \draw[fill] (2,3) circle (0.7pt);
 \draw[fill] (2,4) circle (0.7pt);
   \draw[fill] (1.5,2.5) circle (0.7pt);

   \draw[fill] (1.5,3.5) circle (0.7pt);
      \draw[fill] (1.5,4.5) circle (0.7pt);
         \draw[fill] (3.5,1.5) circle (0.7pt);
            \draw[fill] (4.5,1.5) circle (0.7pt);
            \draw[fill] (4.5,4.5) circle (0.7pt);

 \draw (3,2) --(4,2);
  \draw (4,2) --(5,3);
  \draw (5,4) --(4,4);
 \draw (5,4) --(5,3);
  \draw (1.5,1.5) --(2.5,1.5);
 \draw (2.5,1.5) --(3,2);
 \draw (2,3) --(2,4);
  \draw (2,4) --(3,5);
  \draw (4,5) --(4,4);
 \draw (4,5) --(3,5);
  \draw (1.5,1.5) --(1.5,2.5);
 \draw (1.5,2.5) --(2,3);

   \draw (1.5,4) --(4,4);
   \draw (4,4) --(4,1.5);

 \node [left] at (1.5,3.5) {$2p$};
 \node [below] at (3.5,1.5) {$2p$};
  \node [left] at (1.5,4.5) {$n$};
   \node [below] at (4.5,1.5) {$n$};

\node [below] at (6,1.5) {$a_1$};
\node [left] at (1.5,6) {$a_2$};

\end{tikzpicture}
\end{figure}
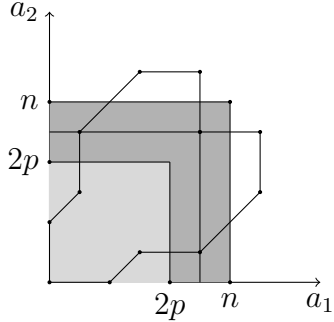

This completes the proof of Proposition \ref{two}.
\end{proof}

\begin{Prop} \label{three} For any $i \in [0,p-1]$ and $a \in \Z^3$ with $a_3 = i$, we have~$a \in T$.
\end{Prop}
\begin{proof}
Let us fix $i \in [0,p-1]$.
Consider a plane and mark on it all integral points~$(a_1,a_2)$ such that $\cO(a_1,a_2, i) \in \mT$.
By definition of~\eqref{eq:enk} all integral points of the polygon in Figure~\ref{figure:six} are marked. The coordinates of its vertices $x_1,\dots,x_{8}$ are listed in the table below.

\begin{figure}[H] 
  \caption{Illustration for Proposition \ref{three}.} \label{figure:six}
  \centering
 \begin{tikzpicture} [scale=0.8]
\path [fill=black!30!white] (1.5,1.5) -- (1.5,0) -- (0,0) -- (0,1.5) -- (1.5,1.5);
\draw [<->] (0,4) -- (0,0) -- (4,0);

 \draw[fill] (0,2) circle (0.7pt);
 \draw[fill] (2,0) circle (0.7pt);
 \draw[fill] (0,0) circle (0.7pt);
  \draw[fill] (2.5,0.5) circle (0.7pt);
 \draw[fill] (0.5,2.5) circle (0.7pt);
   \draw[fill] (2.5,1.5) circle (0.7pt);
 \draw[fill] (1.5,2.5) circle (0.7pt);
  \draw[fill] (1.5,1.5) circle (0.7pt);

 \draw (0,0) --(0,2);
  \draw (0,2) --(0.5,2.5);
  \draw (0.5,2.5) --(1.5,2.5);
 \draw (1.5,2.5) --(1.5,1.5);
  \draw (1.5,1.5) --(2.5,1.5);
 \draw (2.5,1.5) --(2.5,0.5);
 \draw (2.5,0.5) --(2,0);
  \draw (2,0) --(0,0);
  
   \draw (1.5,0) --(1.5,1.5);
 \draw (1.5,1.5) --(0,1.5);
\node [left] at (0,1.5) {$i+p$};

 \node [below left] at (0,0) {$x_1$};
\node [left] at (0,2) {$x_2$};
\node [above ] at (0.5,2.5) {$x_3$};
\node [above right] at (1.5,2.5) {$x_4$};
\node [above right] at (1.5,1.5) {$x_5$};
\node [right] at (2.5,1.5) {$x_6$};
\node [below right] at (2.5,0.5) {$x_7$};
\node [below] at (2,0) {$x_8$};

\node [below] at (4,0) {$a_1$};
\node [left] at (0,4) {$a_2$};
\end{tikzpicture}
\end{figure}

\begin{equation*}
\begin{array}{||c|c|c||}
\hline
  & n=3p & n=3p+1\\
\hline
x_1 & (0,0) & (0,0)\\
x_2 & (0,2p) & (0,2p+1)\\
x_3 & (i,i+2p) & (i,i+2p+1)\\
x_4 & (i+p,i+2p) & (i+p,i+2p+1)\\
x_5 & (i+p,i+p) & (i+p,i+p)\\
x_6 & (i+2p,i+p) & (i+2p+1,i+p)\\
x_7 & (i+2p,i) & (i+2p+1,i)\\
x_8 & (2p,0) & (2p+1,0)\\

\hline
\end{array}\\
\end{equation*}

Our goal is to show that all integral points of the plane are in $T$. 

 We see that $a \in T$ for any $a$ such that $(a_1,a_2) \in [0,i+p]^2, a_3=i$. Since~$i$ is in~$[0,p-1]$,  we have $i+p \ge p$. 

Note that by Propositions \ref{one} and \ref{two} and $S_3$-symmetry of $T$ we have $a \in T$ if~$a_1 \in [p,n]$ or $a_2 \in [p,n]$. Thus we get that $a \in T$ for any $a$ such that $(a_1,a_2)$ is in $[0,n]^2, a_3=i$. In other words, all points in the grey square in Figure $7$ are in $T$. Therefore we can apply Corollary~\ref{bbb2} with $a = (0,0,i)$ and $I = \{1,2\}$. We conclude that if $a_3=i$, then~$a \in T$.

\begin{figure}[H]
  \caption{Illustration for Proposition \ref{three}.}
  \centering
 \begin{tikzpicture} [scale=0.8]
\draw [<->] (0,4) -- (0,0) -- (4,0);

\path [fill=black!30!white] (0,0) -- (3,0) -- (3,3) -- (0,3) -- (0,0);
\path [fill=black!30!white] (0,0) -- (1,0) -- (1,1) -- (0,1) -- (0,0);

 \draw[fill] (0,2) circle (0.7pt);
 \draw[fill] (2,0) circle (0.7pt);
 \draw[fill] (0,0) circle (0.7pt);
  \draw[fill] (2.5,0.5) circle (0.7pt);
 \draw[fill] (0.5,2.5) circle (0.7pt);
   \draw[fill] (2.5,1.5) circle (0.7pt);
 \draw[fill] (1.5,2.5) circle (0.7pt);
  \draw[fill] (1.5,1.5) circle (0.7pt);
        \draw[fill] (3,0) circle (0.7pt);
  \draw[fill] (0,3) circle (0.7pt);
    \draw[fill] (0,1) circle (0.7pt);
      \draw[fill] (1,0) circle (0.7pt);

 \draw (0,0) --(0,2);
  \draw (0,2) --(0.5,2.5);
  \draw (0.5,2.5) --(1.5,2.5);
 \draw (1.5,2.5) --(1.5,1.5);
  \draw (1.5,1.5) --(2.5,1.5);
 \draw (2.5,1.5) --(2.5,0.5);
 \draw (2.5,0.5) --(2,0);
  \draw (2,0) --(0,0);

  \draw (0,1) --(1,1);
   \draw (1,1) --(1,0);
   \draw (0,3) --(3,3);
   \draw (3,0) --(3,3);

\node [left] at (0,1) {$p$};
\node [left] at (0,3) {$n$};
\node [below] at (1,0) {$p$};
\node [below] at (3,0) {$n$};

\node [below] at (4,0) {$a_1$};
\node [left] at (0,4) {$a_2$};
\end{tikzpicture}
\end{figure}

This completes the proof of Proposition \ref{three}.
\end{proof}

\begin{proof}[Proof of Theorem \ref{theorem:x3-3p}]
We combine Propositions \ref{one}--\ref{three} to conclude that if~$a_3$ belongs to $[0,n]$, then~$a \in T$. Therefore we can apply Corollary~\ref{bbb2} with $a = (0,0,0)$ and~$I = \{1,2,3\}$. This concludes the proof of Theorem  \ref{theorem:x3-3p}.
\end{proof}

\subsection{Minimal Lefschetz decomposition for $\cD(X_3^2)$} 
\label{x_3^2}

Consider the case $n=2$, $k=3$. We have $h=n+1=3$, $\dim \K= 3$.
By Proposition~\ref{k=3}, there is no rectangular $S_3$-invariant Lefschetz decomposition of $\cD(X_3^2)$.
In this section we construct a minimal (non-rectangular) $S_3$-invariant Lefschetz decomposition of $\cD(X_3^2)$.
In particular, we prove its fullness.
The same method was used for proving fullness for any $n \not= 2 \text{ mod } 3$.

As we proved in Proposition~\ref{proposition:x32-r0-lower-bound}, an $S_3$-invariant exceptional collection in $\cD(X_3^2)$ 
cannot have less than 13 exceptional objects in the starting block. 

We consider the category $\mB$, generated by $S_3$-orbits of the following line bundles:~$\cO(0,0,0)$, $\cO(1,0,0)$, $\cO(1,1,0)$.
We consider the category $\wb$, generated $\mB$ and~$S_3$-orbit of the line bundle $\cO(2,1,0)$.

Take the collection with the following components:
\begin{equation*}
\begin{array}{||c|c|c||}
\hline
\widehat\mB & \mB(1) & \mB(2) \rule{0em}{2.8ex}\\
\hline
 \cO(0,0,0)&\cO(1,1,1)&\cO(2,2,2)\\
 \cO(1,0,0)&\cO(2,1,1)&\cO(3,2,2)\\
 \cO(1,1,0)&\cO(2,2,1)&\cO(3,3,2)\\
 \cO(2,1,0) & & \\
\hline
\end{array}\\
\end{equation*}

Note that $\wb \subset \wa$ (they differ by $S_3$-orbit of $\cO(2,0,0)$) and $\mA \subset \mB$ (they differ by $S_3$-orbit of $\cO(1,1,0)$). In other words, we remove one orbit from $\wa$ and add one orbit to $\mA$. 

The starting component $\widehat\mB$ is generated by $1 + 3 + 3 + 6 = 13$ line bundles, 
while the other two components are generated by $1 + 3 + 3 = 7$ line bundles. Evidently,~$\mB \subset \wb$.

\begin{Th} 
\label{theorem:x32}
The categories $\wb$, $\mB(1)$ and $\mB(2)$ described above generate a minimal~$S_3$-invariant Lefschetz collection in $\cD(X_3^2)$.
In particular,
 
\begin{equation}\label{decomp23}
\cD(X_3^2) = \langle \wb,\mB(1),\mB(2) \rangle.
\end{equation}
\end{Th}

\begin{proof} Obviously, the categories $\wb$, $\mB(1)$ and $\mB(2)$ are $S_3$-invariant. 

Let us prove that $(\wb,\mB(1),\mB(2))$ is semiorthogonal. Since $\wb \subset \wa$ and $\mB =  \langle \mA, \mC \rangle$, where $\mA$ and $\wa$ are the components of \eqref{collection} and $\mC$ is the category generated by the~$S_3$-orbit of $\cO(1,1,0)$, it is enough to check that

\begin{align*}
\Ext^\bullet(\mC(1),\wb)=0,\\
\Ext^\bullet(\mC(2),\wb)=0,\\
\Ext^\bullet(\mC(2),\mC(1))=0.
\end{align*}

These equalities can be easily checked by inspection using \eqref{eq:semiorthogonality}.

We conclude that $(\wb,\mB(1),\mB(2))$ is~$S_3$-invariant and semiorthogonal.
Let us show that it generates $\cD(X_3^2)$.

For this we denote by $\mT$ the triangulated subcategory of~$\cD(X_3^2)$ generated by~the categories $\wb, \mB(1), \mB(2)$.
Applying Corollary~\ref{bbb2} several times we will show that more line bundles are contained in $\mT$.
We note $\mT$ is $S_3$-invariant, so as soon as a line bundle is proved to be contained in $\mT$, its entire $S_3$-orbit is also contained in $\mT$.

{\bf Step 1.} We note that $\cO(2,2,1)$, $\cO(2,2,2)$, and $\cO(2,2,3)$ are all in $\mT$ 
(the first is in $\mB(1)$, while the other two are in $\mB(2)$).
Applying Corollary~\ref{bbb2} with $a = (2,2,1)$ and~$I = \{3\}$ we conclude that all line bundles $\cO(2,2,t)$ are in $\mT$.
In particular, 
\begin{equation*}
\cO(2,2,0) \in \mT.
\end{equation*}

{\bf Step 2.} We note that $\cO(1,2,0)$, $\cO(1,2,1)$, and $\cO(1,2,2)$ are in $\mT$
(the first is in $\wb$, while the other two are in $\mB(1)$).
Applying Corollary~\ref{bbb2} with $a = (1,2,0)$ and~$I = \{3\}$ we conclude that all line bundles $\cO(1,2,t)$ are in $\mT$.
In particular, 
\begin{equation*}
\cO(1,2,3) \in \mT.
\end{equation*}

{\bf Step 3.} We note that $\cO(3,2,1)$, $\cO(3,2,2)$, and $\cO(3,2,3)$ are in $\mT$ (for the first of them we use the result of Step 2).
Applying Corollary~\ref{bbb2} with $a = (3,2,1)$ and~$I = \{3\}$ we conclude that all line bundles $\cO(3,2,t)$ are in $\mT$.
In particular, 
\begin{equation*}
\cO(3,2,0) \in \mT.
\end{equation*}

{\bf Step 4.} We note that $\cO(2,0,1)$, $\cO(2,0,2)$, and $\cO(2,0,3)$ are in $\mT$ 
(for the last two of them we use the results of Step 1 and Step 3 and $S_3$-invariance of $\mT$).
Applying Corollary~\ref{bbb2} with~$a = (2,0,1)$ and $I = \{3\}$ we conclude that all line bundles~$\cO(2,0,t)$ are in $\mT$.
In particular, 
\begin{equation*}
\cO(2,0,0) \in \mT.
\end{equation*}

Combining the original collection with the results of Steps 1--4 above and $S_3$-invariance, we see that all line bundles $\cO(a)$ with $a \in [0,2]^3$ are contained in $\mT$.
Therefore, by Theorem~\ref{generate} we have $\mT = \cD(X_3^2)$.

Finally, the minimality of the constructed Lefschetz collection follows from Proposition~\ref{proposition:x32-r0-lower-bound}.
\end{proof}

\begin{Def}[{\cite[Definition~1.3]{KuznetsovSmirnov}}]
The \textit{rectangular part} of Lefschetz decomposition~$\langle \mB_0, \mB_1(1), \dots, \mB_d(d) \rangle = \cD(X)$ is $(\mB_d, \mB_d(1), \dots, \mB_d(d))$. 
The subcategory of~$\cD(X)$ orthogonal to the rectangular part of a given Lefschetz decomposition is called its \textit{residual category}:
\begin{equation*}
\mathcal{R}_{\mB^\bullet}=\langle \mB_d, \mB_d(1), \dots, \mB_d(d) \rangle^{\bot}.
\end{equation*}
\end{Def}

\begin{Th}
The residual category of the Lefschetz decomposition \eqref{decomp23} is generated by $S_3$-orbit of the line bundle $\cO(1,-1,0)$.
\end{Th}
\begin{proof}
Denote by~$\mathcal{R}$ the category generated by $S_3$-orbit of the line bundle $\cO(1,-1,0)$. Firstly, we need to check that

\begin{align*}
\Ext^\bullet(\mB,\mathcal{R})=0,\\
\Ext^\bullet(\mB(1),\mathcal{R})=0,\\
\Ext^\bullet(\mB(2),\mathcal{R})=0.
\end{align*}

These equalities can be easily checked by inspection using \eqref{eq:semiorthogonality}.

Secondly, we prove that $\wb \subset \langle \mathcal{R}, \mB\rangle$. Clearly, for that it is enough to prove that the line bundle~$\cO(2,1,0)$ is in $\langle \mathcal{R}, \mB\rangle$. Indeed, we note that  $\cO(-1,1,0)$, $\cO(0,1,0)$, and $\cO(1,1,0)$ are in $\langle \mathcal{R}, \mB\rangle$. Applying Corollary~\ref{bbb2} with~$a = (-1,1,0)$ and $I = \{1\}$ we conclude that the line bundle~$\cO(2,1,0)$ is in $\langle \mathcal{R}, \mB\rangle$.

Thus
\begin{equation*}
\langle \mathcal{R},\mB, \mB(1), \mB(2) \rangle \supset  \langle \wb,\mB(1),\mB(2) \rangle = \cD(X_3^2).
\end{equation*}

Therefore $\mathcal{R}$ is the residual category of the Lefschetz decomposition \eqref{decomp23}.
\end{proof}

\end{document}